\documentclass{article}

\usepackage{alltt}
\usepackage{amsfonts}
\usepackage{amsmath}
\usepackage{amssymb}
\usepackage{amsthm}
\usepackage{booktabs}
\usepackage{caption}
\usepackage{subcaption}
\usepackage{enumitem}
\usepackage{fancyhdr}
\usepackage{fullpage}
\usepackage{graphicx}
\usepackage{mathdots}
\usepackage{mathtools}
\usepackage{microtype}
\usepackage{multirow}
\usepackage{pdflscape}
\usepackage{pgfplots}
\usepackage{siunitx}
\usepackage{slashed}
\usepackage{tabularx}
\usepackage{tikz}
\usepackage{tkz-euclide}
\usepackage[normalem]{ulem}
\usepackage[all]{xy}
\usepackage{imakeidx}
\usepackage{verbatim}
\usepackage[
backend=bibtex,
style=numeric,
maxbibnames = 10,
]{biblatex}
\usepackage{url}
\usepackage{breakurl}
\usepackage[breaklinks]{hyperref}

\usetikzlibrary{arrows}
\usetikzlibrary{arrows.meta}
\usetikzlibrary{decorations.markings}
\usetikzlibrary{decorations.pathmorphing}
\usetikzlibrary{positioning}
\usetikzlibrary{fadings}
\usetikzlibrary{intersections}
\usetikzlibrary{cd}

\pgfarrowsdeclarecombine{twolatex'}{twolatex'}{latex'}{latex'}{latex'}{latex'}
\tikzset{->/.style = {decoration={markings,
			mark=at position 1 with {\arrow[scale=2]{latex'}}},
		postaction={decorate}}}
\tikzset{<-/.style = {decoration={markings,
			mark=at position 0 with {\arrowreversed[scale=2]{latex'}}},
		postaction={decorate}}}
\tikzset{<->/.style = {decoration={markings,
			mark=at position 0 with {\arrowreversed[scale=2]{latex'}},
			mark=at position 1 with {\arrow[scale=2]{latex'}}},
		postaction={decorate}}}
\tikzset{->-/.style = {decoration={markings,
			mark=at position #1 with {\arrow[scale=2]{latex'}}},
		postaction={decorate}}}
\tikzset{-<-/.style = {decoration={markings,
			mark=at position #1 with {\arrowreversed[scale=2]{latex'}}},
		postaction={decorate}}}
\tikzset{->>/.style = {decoration={markings,
			mark=at position 1 with {\arrow[scale=2]{latex'}}},
		postaction={decorate}}}
\tikzset{<<-/.style = {decoration={markings,
			mark=at position 0 with {\arrowreversed[scale=2]{twolatex'}}},
		postaction={decorate}}}
\tikzset{<<->>/.style = {decoration={markings,
			mark=at position 0 with {\arrowreversed[scale=2]{twolatex'}},
			mark=at position 1 with {\arrow[scale=2]{twolatex'}}},
		postaction={decorate}}}
\tikzset{->>-/.style = {decoration={markings,
			mark=at position #1 with {\arrow[scale=2]{twolatex'}}},
		postaction={decorate}}}
\tikzset{-<<-/.style = {decoration={markings,
			mark=at position #1 with {\arrowreversed[scale=2]{twolatex'}}},
		postaction={decorate}}}

\tikzset{circ/.style = {fill, circle, inner sep = 0, minimum size = 3}}
\tikzset{scirc/.style = {fill, circle, inner sep = 0, minimum size = 1.5}}
\tikzset{mstate/.style={circle, draw, blue, text=black, minimum width=0.7cm}}

\tikzset{eqpic/.style={baseline={([yshift=-.5ex]current bounding box.center)}}}
\tikzset{commutative diagrams/.cd,cdmap/.style={/tikz/column 1/.append style={anchor=base east},/tikz/column 2/.append style={anchor=base west},row sep=tiny}}

\theoremstyle{definition}

\newtheorem{nthm}{Theorem}[section]

\newtheorem{conjecture}[nthm]{Conjecture}
\newtheorem{nlemma}[nthm]{Lemma}
\newtheorem{nprop}[nthm]{Proposition}

\newtheorem{ncor}[nthm]{Corollary}

\newtheorem{defi}[nthm]{Definition}
\newtheorem{eg}[nthm]{Example}

\newtheorem{remark}[nthm]{Remark}
\newtheorem{intthm}{Theorem}

\newcommand{\al}{\alpha}
\newcommand{\ga}{\gamma}

\newcommand{\MD}{\mathcal{D}}

\newcommand{\CC}{\mathbb{C}}
\newcommand{\QQ}{\mathbb{Q}}
\newcommand{\RR}{\mathbb{R}}
\newcommand{\ZZ}{\mathbb{Z}}

\newcommand{\LX}{\mathcal{L}X}
\newcommand{\E}{\mathcal{E}}

\newcommand{\KK}{\mathbb{K}}

\newcommand{\MN}{\mathcal{N}}
\newcommand{\kk}{\mathbf{k}}

\def\st{\bgroup \ULdepth=-.55ex \ULset}

\begin{document}
	
	\title{On the quantum differential equations for a family of non-K\"ahler monotone symplectic manifolds}
	
	\date{\today}
	
	\author{Kai Hugtenburg}
	\maketitle
	\begin{abstract}
		In this paper we prove Gamma Conjecture $1$ for twistor bundles of hyperbolic $6$ manifolds, which are monotone symplectic manifolds which admit no K\"ahler structure. The proof involves a direct computation of the $J$-function, and a version of Laplace's method for estimating power series (as opposed to integrals). This method allows us to rephrase Gamma Conjecture $1$ in certain situations to an Ap\'ery-like discrete limit. We use this to give a simple proof of Gamma Conjecture $1$ for projective spaces. Additionally we show that the quantum connections of the twistor bundles we consider have unramified exponential type.
	\end{abstract}
	\section{Introduction}
	\begin{defi}
		The Gamma class of an almost complex manifold is defined by: \begin{equation}
			\Gamma_X := \exp \left( -C_{eu}c_1 + \sum_{k\geq 2} (-1)^k (k-1)!\zeta(k) ch_k\right),
		\end{equation}
		where $ch_k$ denotes the degree $2k$ part of the Chern character of $TX$, $\zeta$ denotes the Riemann-zeta function, and $C_{eu}$ is the Euler-Mascheroni constant.
	\end{defi}
	The $J$-function of a monotone symplectic manifold is a cohomology valued function which captures information about the quantum differential equation of $X$, see Section \ref{sect: quantum differential equations} for a precise definition. Gamma Conjecture $1$ concerns the restriction of the $J$-function to the anti-canonical line. This is an element $J(t) \in H^*(X)[[t]][\log t]$, which can be written as \begin{equation}
		J(t) = e^{c_1 \log(t)}\sum_{n \geq 0} J_{rn}t^{rn},
	\end{equation}
	for cohomology classes $J_{rn} \in H^*(X)$, where $r \in \mathbb{N}$ is such that $c_1: H_2(X;\ZZ) \rightarrow \ZZ$ is generated by $r$. Following Galkin-Golyshev-Iritani \cite{GGI}:
	\begin{defi}
		We say that $X$ satisfies Gamma Conjecture $1$ if:
		\begin{equation}
			\lim_{t \to \infty} \frac{J(t)}{\langle J(t), pt \rangle} = \Gamma_X.
		\end{equation}
	\end{defi}
	In \cite{GGI} this is conjectured to hold for Fano varieties, and proven for projective spaces and Grassmanians. In \cite{GoZ} it is proven for Fano threefolds of Picard rank 1. In \cite{GI} it is shown that Gamma Conjecture $1$ is compatible with taking (Fano) hyperplane sections. They also show that it holds for Fano toric varieties under the additional assumption of a mirror analogue of property $\mathcal{O}$. In \cite{HuKeLiYang} the conjecture is proven for Del Pezzo surfaces. 
	\begin{remark}
		Property $\mathcal{O}$ states that the operator $c_1 \star: H^*(X;\CC) \rightarrow H^*(X;\CC)$ is such that the eigenvalue with largest modulus lies on the real line, and moreover is a simple eigenvalue (eigenspace of dimension $1$). For varieties satisfying property $\mathcal{O}$, \cite{GGI} prove an equivalence between the above version of the Gamma conjecture, and a statement in terms of the flat section of smallest growth near $u = 0$ for the quantum connection $\nabla_{\partial_u}$. The twistor bundles we consider do no satisfy property $\mathcal{O}$, so we will only talk about the version of the Gamma conjecture stated above.
	\end{remark}
	
	Now, let $M$ be a hyperbolic $2n$-dimensional manifold with vanishing Stiefel-Whitney classes, and $Z$ the associated `twistor space', that is, the total space of the bundle of orthogonal complex structures on the tangent spaces of $M$. These were considered by Reznikov \cite{Rez}, \cite{FinePanov} and \cite{EEllsSalamon}, where it is shown that for $n \geq 3$, the twistor bundles $Z$ are monotone. Moreover, they are not K\"ahler. In \cite{Evans} the quantum cohomology ring for such $Z$ is computed for $n=3$, which we recall in Theorem \ref{thm: quantum cohomology of FP}. In this paper we prove: 
	\begin{intthm}
		\label{thm: FP satisfies Gamma 1}
		Let $Z$ be the twistor bundle associated to a hyperbolic $6$ manifold with vanishing Stiefel-Whitney classes. Then $Z$ satisfies Gamma Conjecture $1$.
	\end{intthm}
	Additionally, in Section \ref{sect: FP has unram} we prove: 
	\begin{intthm}
		\label{thm: FP has unramified exponential type}
		Let $Z$ be the twistor bundle of a hyperbolic $6$ manifold with vanishing Stiefel-Whitney classes, then the associated quantum connection $\nabla_{\frac{d}{du}}$ is of unramified exponential type. 
	\end{intthm}
	\begin{remark}
		As far as the author is aware, this is the first non-K\"ahler symplectic manifold for which either Gamma Conjecture $1$ or the unramified exponential type of the quantum connection has been verified.
	\end{remark}
	\begin{remark}
		The author apologises for having erroneously stated in conversations that the connection would not be of unramified exponential type. Hopefully the record will be corrected through this publication.
	\end{remark}
	It is thus natural to ask whether Gamma Conjecture $1$ (and also the unramified exponential type conjecture) hold for all monotone symplectic manifolds, and not just the Fano ones. We observe: 
	\begin{nlemma}
		Let $X$ be a monotone symplectic manifold satisfying Gamma Conjecture $1$, then the (very small) quantum ring of $X$ is deformed (i.e. not isomorphic to the classical cohomology ring).
	\end{nlemma}
	Thus, the validity of Gamma Conjecture $1$ would imply the existence of a pseudo-holomorphic sphere in degree $\beta \in H_2(X)$ with $c_1(\beta) \leq 2dim_\CC X$.
	\subsection{Rephrasing Gamma Conjecture $1$}
	To proof Theorem \ref{thm: FP satisfies Gamma 1}, we rephrase Gamma Conjecture $1$ to a statement about the limit of the coefficients of the $J$-function. To make this precise, first define:
	\begin{defi}
		\label{defi: subpolynomially increasing sequence}
		Say a sequence $b_n$ `increases subpolynomially' if $b_n$ is increasing, $b_n \leq Bn$ for some $B > 0$, and $\lim_{n \to \infty}\frac{b_n}{n^p} = 0$ for all $p > 0$. 
	\end{defi}
	\begin{defi}
		\label{defi: superpolynomially peaked series}
		Let $I(x) = \sum_{n = 0}^\infty a_nx^n$ be a series which convergences absolutely on $(-\infty,\infty)$. We say it is `superpolynomially peaked near $n = f(x)$' if, for every subpolynomially increasing sequence $b_n$, and for every integer $k \geq 0$ we have that \begin{equation}
			\lim_{x \to \infty} \frac{\sum_{n \geq 1}^\infty a_nb_n\log(n)^kx^n -\sum_{n \geq 1}^\infty a_nb_nx^n\log(f(x))^k}{I(x)} = 0.
		\end{equation}
	\end{defi}
	We then show:
	\begin{intthm}
		\label{thm: rephrasing Gamma Conjecture 1}
		Suppose that the quantum period (given by $\langle J(t), pt \rangle$) of a monotone symplectic manifold $X$ has non-negative coefficients and is superpolynomially peaked near $n = Ct$, and that the sequences $\frac{\langle J_{rn}, \alpha \rangle}{\langle J_{rn}, pt \rangle}$ are subpolynomially increasing for any $\alpha \in H_*(X)$, then Gamma Conjecture $1$ is equivalent to the Ap\'ery limit: \begin{equation}
			\lim_{n \to \infty} \frac{e^{c_1 \log(\frac{n}{C})}J_{rn}}{\langle J_{rn}, pt \rangle} = \Gamma_X.
		\end{equation}
	\end{intthm}
	In Section \ref{sect: Gamma Conjecture FP} we verify these assumptions for the twistor bundles and then directly compute the required Ap\'ery limit to conclude Gamma Conjecture $1$. In Section \ref{sect: Gamma Conjecture for projective space} we prove that $\mathbb{CP}^n$ satisfies these assumptions, and use this to give a simple proof of Gamma Conjecture $1$ for projective spaces.
	\begin{remark}
		We expect that the quantum period is superpolynomially peaked near $n = \frac{T}{r}t$, where $T = \max \{|\lambda| \mid \lambda \text{ is an eigenvalue of } c_1 \star: QH^*(X) \rightarrow QH^*(X) \}$ and $r$ is as before.
	\end{remark}
	To verify the assumptions of the above theorem for projective spaces and for the twistor bundles over hyperbolic manifolds, we prove (using ideas from \cite{Paris}):
	\begin{intthm}
		\label{thm: hypergeometric series are superpolynmoially peaked}
		The series
		\begin{equation}
			I(x) := \sum_{n \geq 0} \frac{\prod_{r=1}^{p} \Gamma(\alpha_rn+a_r)}{\prod_{r = 1}^{q} \Gamma(\beta_rn+b_r)}x^n
		\end{equation}
		is superpolyonmially peaked near $n = (hx)^{\frac{1}{\kappa}}$, where $\kappa = \sum_r \beta_r - \sum_r \alpha_r$ (which we assume satisfies $\kappa \geq 1$), and $h = \prod_{r}\alpha_r^{\alpha_r} \prod_r \beta_r^{-\beta_r}$.
	\end{intthm}
	\begin{remark}
		Combing the computation of the quantum period of Fano threefolds in \cite{Coat} with Theorem \ref{thm: hypergeometric series are superpolynmoially peaked} we find that the quantum periods of Fano threefolds of Picard rank $1$ are superpolynomially peaked near $n = \frac{T}{r}t$. Moreover, computer calculations show that the quantum periods of Fano threefolds with higher Picard rank are also superpolynomially peaked. One might thus expect the conclusion of Theorem \ref{thm: rephrasing Gamma Conjecture 1} to hold more generally. Certainly, one would hope to weaken the assumption of non-negative coefficients.
	\end{remark}
	\begin{remark}
		During the writing of this manuscript, we learned that ChunYin Hau, a PhD student supervised by Hiroshi Iritani, has obtained related results on the asymptotics of the coefficients $J_{rn}$. 
	\end{remark}

	\subsection{Acknowledgements}
	The author would like to thank Paul Seidel, for suggesting twistor bundles as an interesting space for computations, Hiroshi Iritani and his student ChunYin Hau, for a helpful conversation on the Gamma Conjecture, and Jonny Evans for explaining various parts of his computation of the quantum cohomology of the twistor bundles. This work was partly supported by the ERC Starting Grant 850713 – HMS and EPSRC Grant EP/W015749/1.
	
	\section{The Gamma class and the loop space}
	Here we show how the Gamma class of any almost complex manifold $X$ can be obtained as a certain renormalised limit of the $S^1$-equivariant Euler class of the positive normal bundle of $X$ inside the loop space $\LX$. This was already observed in \cite{Lu}, but we take a slightly different approach which is better adapted for our proofs of Gamma Conjecture $1$. To this end, let $\MN_{X \subset \LX} \rightarrow X$ denote the (infinite dimensional) normal bundle. This comes with a natural $S^1$-action given by rotation of loops. One can decompose this bundle into Fourier modes, to obtain an isomorphism \begin{equation}
		\MN_{X \subset \LX} \cong \bigoplus_{d \in \ZZ}TX \otimes \eta^d,
	\end{equation}
	where $\eta^d$ denotes the one-dimensional $S^1$ representation with weight $d$, and $S^1$ acts trivially on the tangent bundle of $X$. Then define the degree $n$ approximation to the positive part by \begin{equation}
		\MN^n := \bigoplus_{k = 0}^n TX \otimes \eta^k.
	\end{equation}
	A short computation shows that \begin{equation}
		e_{S^1}(\MN^n) = \prod_{k = 1}^{n} \prod_{j =1}^{dim(X)}  (uk + \delta_j),
	\end{equation}
	where the $\delta_j$ are the Chern roots of (the tangent bundle of) $X$ and $u$ is the equivariant parameter (coming from $H^2(BU(1))$). Inverting the $S^1$-equivariant Euler-class yields:
	\begin{equation}
		e_{S^1}(\MN^n)^{-1} = \frac{1}{(un!)^{dim(X)}}\prod_{k = 1}^{n}  \prod_{j =1}^{dim(X)}  (1 + \frac{\delta_j}{uk})^{-1}.
	\end{equation}
	Normalise this cohomology class to have constant term $1$, that is, consider the class $\frac{e_{S^1}(\MN^n)^{-1}}{\langle e_{S^1}(\MN^n)^{-1}, pt \rangle}$.
	Then, the approach by \cite{Lu} is to restrict to $u = 1$ and consider the limit \begin{equation}
		\lim_{d \to \infty} \prod_{k = 1}^{n}  \prod_{j =1}^{dim(X)}  (1 + \frac{\delta_j}{k})^{-1} e^{\frac{\delta_j}{k}} = e^{\ga c_1} \Gamma_X.
	\end{equation}
	Where the equality follows from the Weierstrass form of the $\Gamma$ function, and $c_1 = c_1(TX)$ is the first Chern class of $X$. Instead, we rearrange this infinite product differently:
	\begin{equation}
		\prod_{k = 1}^{n} \prod_{j =1}^{dim(X)}  (1 + \frac{\delta_j}{uk})^{-1} e^{\frac{\delta_j}{uk}} =  e^{\frac{c_1(\sum_{k=1}^{n}\frac{1}{k} - log(n))}{u}}e^{\frac{c_1log(d)}{u}} \prod_{k = 1}^{n} \prod_{j =1 }^{dim(X)}  (1 + \frac{\delta_j}{uk})^{-1}.
	\end{equation}
	The first term on the right-hand side converges to $e^{\ga c_1}$. So that we find: \begin{nlemma}
		we have: 
		\begin{equation}
			\label{eq: limit of Phid}
			\lim_{n \to \infty} e^{c_1log(n)} \frac{e_{S^1}(\MN^n)^{-1}}{\langle e_{S^1}(\MN^n)^{-1}, pt \rangle}\bigg\vert_{u = 1} = \Gamma_X
		\end{equation}
	\end{nlemma} 
	\section{Quantum differential equations}
	\label{sect: quantum differential equations}
	Let $(X,\omega)$ be a monotone symplectic manifold, which means that $[\omega] = \lambda c_1 \in H^2(X;\RR)$ for some $\lambda \in \RR_{> 0}$, where $c_1$ is the first Chern class (of the tangent bundle) of $X$. Associated to an element $\ga \in H^*(X)$, one can (assuming convergence), form the bulk-deformed quantum product $\star_\gamma$, defined by: \begin{equation}
		\left\langle \al_1 \star_{\gamma} \al_2, \al_3 \right\rangle_X = \sum_{\substack{\beta \in H_2(X)\\l \geq 0}} GW_{0,3+l,\beta}(\al_1,\al_2, \al_3, \ga^{\otimes l})
	\end{equation}
	where $GW_{0,k}$ denotes a genus $0$ Gromov-Witten invariant, see e.g. \cite{MS12}, and $\left\langle \_, \_ \right\rangle_X$ denotes the Poincar\'e pairing on $X$. Now consider the quantum cohomology with bulk deformation restricted to the anticanonical line $\log(t)c_1$. This means that the quantum product on $H^*(X)[[t]]$ is given by \begin{equation}
		\left\langle \al_1 \star_{\log(t)c_1} \al_2, \al_3 \right\rangle_X = \sum_{\beta \in H_2(X)} GW_{0,3,\beta}(\al_1,\al_2, \al_3)t^{c_1(\beta)}.
	\end{equation}
	We will also use the `very small' quantum product denoted by $\star := \star_0$, obtained by setting $t = 1$. One can also consider the quantum connection, the particular ones we need are:
	\begin{equation}
		\nabla_{\log(t)c_1}: H^*(X)[[t]]((u)) \rightarrow H^*(X)[[t]]((u))
	\end{equation}
	defined by
	 \begin{equation}
		\nabla_{\log(t)c_1}(\alpha) = t\partial_t \alpha - \frac{c_1 \star_{\log(t)c_1} \al}{u}
	\end{equation}
	and the connection \begin{equation}
		\nabla_{\frac{d}{du}}: H^*(X)((u)) \rightarrow H^*(X)((u))
	\end{equation}
	given by \begin{equation}
		\nabla_{\frac{d}{du}}(\alpha) = \frac{d\alpha}{du} + \frac{\mu(\alpha)}{u} + \frac{c_1 \star \alpha}{u^2},
	\end{equation}
	where $\mu: H^p(X) \rightarrow H^p(X)$ is given by $\frac{p - dim_\CC X}{2}$.
	
	The restriction of the J-function to the anti-canonical line, which we will denote by $J(t) \in H^*(X)[[t]][\log(t)]$, is defined as follows. First, let \begin{equation}
		S: H^*(X) \rightarrow H^*(X)[[t]][\log t]((u))
	\end{equation}
	be the map that associates to $\al \in H^*(X)$ the flat (under $\nabla_{\log(t)c_1}$) extension $S(\al)$. Extend $S$ linearly in $t$ and $u$, and then define $\mathcal{J} = S^{-1}$. Finally, define \begin{equation}
		J(t) := \mathcal{J}(1)|_{u = 1}.
	\end{equation}
	From the definition it is clear that for any differential operator $P(t, t\frac{d}{dt})$, we have: \begin{equation}
		P(t, t\frac{d}{dt}) J(t) = 0 \iff P(t, \nabla_{\log(t)c_1}|_{u = 1}) 1 = 0
	\end{equation}
	Using the topological recursion relations (see e.g. \cite[Lemma~10.2.2]{CK}), one obtains a different description of $J(t)$. To this end, fix a basis $\phi_1, \dots, \phi_N \in H^*(X)$. Let $\phi^1, \dots, \phi^N$ be the dual basis (under the Poincar\'e pairing). We then have: \begin{equation}
		J(t) = e^{c_1\log(t)}\left(1 + \sum_{i = 1}^N \sum_{\beta \in H_2(X)\setminus \{ 0 \}} \left\langle \frac{\phi_i}{1 - \psi} \right\rangle t^{c_1(d)} \phi^i  \right). 
	\end{equation}
	Here $\frac{1}{1-\psi}$ should be expanded as $1 + \psi + \psi^2 + \dots$, and the angle brackets denote a descendent Gromov-Witten invariant (see e.g \cite{RT}).
	\section{Laplace's method for power series}
	\label{sect: Laplaces method for power series}
	Recall that Laplace's method provides an asymptotic expansion (for large $x$) for functions of the form \begin{equation}
		I(x) = \int_{a}^{b} e^{xf(t)}dt,
	\end{equation}
	whenever $f$ has a unique maximum in $[a,b]$. The main idea is that the largest contribution to the integral comes from a neighbourhood of the maximum of $f$. This relies on the fact that one can exponentially bound the integrand outside of a neighbourhood of the maximum.
	
	For functions defined by a power series \begin{equation}
		I(x) = \sum_{n \geq 0} a_n x^n
	\end{equation}
	such that the terms $a_n x^n$ similarly show a strong peak for some value $N$ (which might depend on $x$), the analogue of Laplace's method has been much less studied. Stokes \cite{Stokes} considers the function \begin{equation}
		F(x) := \sum_{n \geq 0} \frac{\prod_{r=1}^{p} \Gamma(n+a_r)}{\prod_{r = 1}^{q} \Gamma(n+b_r)}x^n
	\end{equation}
	which converges for positive real numbers $a_r$ and $b_r$ with $q \geq p+1$. He obtains (without careful analysis) an asymptotic expansion \begin{equation}
		F(x) \sim (2\pi)^{(1-\kappa)/2}\kappa^{-1/2} x^{(\theta + \frac{1}{2})/\kappa}\exp(\kappa x^{1/\kappa})
	\end{equation}
	where $\kappa = q - p$, $\theta = \sum_{r=1}^p a_r - \sum_{r=1}^q b_r + \frac{\kappa}{2}$ and $f(x) \sim g(x)$ means that $\lim_{x \to \infty} \frac{f(x)}{g(x)} = 1$. A careful proof of this was only given more than a century later by Paris \cite{Paris}. Another reference where a special case of the above problem is considered (again without any proofs) is in \cite[Section~6.7]{Ben}. More recently \cite{Gerhold} uses methods similar to \cite{Paris} for Mathieu series. 
	
	To make the analogue with an exponential integrand precise, we define:
	\begin{defi}
		\label{defi: exponentially bounded series}
		Say a series $I(x) = \sum_{n \geq 0} a_n x^n$ is \emph{exponentially bounded away from $n = f(x)$} if there exists a polynomial $f(x)$, and a choice of scaling parameter $\nu > 0$, such that for all subpolynomially increasing sequences $b_n$, we have \begin{equation}
			\frac{\sum_{n = 0}^{N_-} a_n b_n x^n}{I(x)} = O(\exp{(-\al x^{\beta})})
		\end{equation}
		and 
		\begin{equation}
			\frac{\sum_{n = N_+}^{\infty} a_n b_n x^n}{I(x)} = O(\exp{(-\al x^{\beta})})
		\end{equation}
		for some $\al, \beta> 0$, and $N_\pm = \lfloor f(x)(1 \pm \epsilon) \rfloor$, where $\epsilon = x^{-\nu}$.
	\end{defi}
	Also recall the definition of superpolynomially peaked series:
	\begin{defi}
		Let $I(x) = \sum_{n = 0}^\infty a_nx^n$ be a series which convergences absolutely on $(-\infty,\infty)$. We say it is `superpolynomially peaked near $n = f(x)$' if, for every subpolynomially increasing sequence $b_n$, and for every integer $k \geq 0$ we have that \begin{equation}
			\lim_{x \to \infty} \frac{\sum_{n \geq 1}^\infty a_nb_n\log(n)^kx^n -\sum_{n \geq 1}^\infty a_nb_nx^n\log(f(x))^k}{I(x)} = 0.
		\end{equation}
	\end{defi}
	We then show:
	\begin{nthm}
		\label{thm: exponentially bounded implies superpolynomially peaked}
		Let $I(x) = \sum_{n\geq 0} a_nx^n$ be a series which is exponentially bounded away from $n = f(x)$, then it is superpolynomially peaked near $n = f(x)$.
	\end{nthm}
	\begin{proof}
		Let $b_n$ be any subpolynomially increasing sequence. Then so is $\log(n)^kb_n$ for all $k$. Let $G_k(x) = \sum_{n \geq 1}^\infty a_nb_n\log(n)^kx^n$. As $I(x)$ is exponentially bounded away from $n = f(x)$ we have: \begin{equation}
			\frac{G_k(x) - \sum_{n = N_-}^{N_+} a_n b_n \log(n)^k x^n}{I(x)} = O(\exp{(-\al x^{\beta})})
		\end{equation}
		for some $\al, \beta > 0$, where $G_k(x) = \sum_{n = 0}^{\infty} a_n b_n \log(n)^k x^n$, and $N_\pm$ are as in Definition \ref{defi: exponentially bounded series}
		We thus have \begin{align}
			\left\lvert \frac{G_k(x) - G_0(x)log(f(x))^k}{I(x)} \right\rvert &=  \left\lvert \frac{\sum_{n = N_-}^{N_+} a_n b_n x^n \left(\log(n)^k - log(f(x))^k \right)}{I(x)} \right\rvert + O(\exp{(-\al x^{\beta})})\\
			&\leq  \frac{\sum_{n = N_-}^{N_+} a_n b_n x^n \left\lvert\log(n)^k - log(f(x))^k \right\rvert}{I(x)}  + O(\exp{(-\al x^{\beta})}).
		\end{align}
		Now for $N_- < n < N_+$ we have: 
		\begin{equation}
			\left\lvert\log(n)^k - \log(f(x))^k \right\rvert \leq \left\lvert \log(f(x)(1\pm \epsilon))^k - \log(f(x))^k \right\rvert \leq \sum_{r = 1}^{k} \binom{k}{r} \left\lvert\log(1 \pm \epsilon)\right\rvert^r \log(f(x))^{k-r} = O(x^{-\nu})
		\end{equation}
		where the last relation holds as $\log(1 \pm \epsilon) = O(\epsilon)$ for small $\epsilon$, and $f(x)$ is a polynomial. Then, in the range $N_- < n < N_+$ we have \begin{equation}
		b_n \left\lvert\log(n)^k - log(f(x))^k \right\rvert < b_{N_+}x^{-\nu} =  O((1+ \epsilon)^p x^{p - \nu}) = O(x^{p- \nu}) \text{ for all } p >0,
		\end{equation}
		as $b_n$ increases subpolynomially, so that $b_{N_+} = O(N_+^p)$ for all $p > 0$. Taking $0 < p < \nu$ shows that \begin{equation}
			\frac{\sum_{n = N_-}^{N_+} a_nb_n x^n \left\lvert\log(n)^k - log(f(x))^k \right\rvert}{I(x)} < \frac{\sum_{n = N_-}^{N_+}  a_n x^n}{I(x)} x^{-r} < x^{-r},
		\end{equation}
		for all $\nu >r >0$, and thus: \begin{equation}
		\left\lvert \frac{G_k(x) - G_0(x)log(x)^k}{I(x)} \right\rvert < x^{-r},
		\end{equation}
		as required.
	\end{proof}
	The goal is to show: \begin{nthm}
		The series:
		\begin{equation}
			F(x) := \sum_{n \geq 0} \frac{\prod_{r=1}^{p} \Gamma(\alpha_rn+a_r)}{\prod_{r = 1}^{q} \Gamma(\beta_rn+b_r)}x^n
		\end{equation}
		is superpolynomially peaked around $n = (hx)^{\frac{1}{\kappa}}$, where $\kappa = \sum_r \beta_r - \sum_r \alpha_r$ (which we assume satisfies $\kappa \geq 1$), and $h = \prod_{r}\alpha_r^{\alpha_r} \prod_r \beta_r^{-\beta_r}$.
	\end{nthm}
	By theorem \ref{thm: exponentially bounded implies superpolynomially peaked}, it suffices to show $F(x)$ is exponentially bounded away from $n = (hx)^{\frac{1}{\kappa}}$. Paris showed this result when the series $b_n = 1$:
	\begin{nlemma}[{\cite[Section~3]{Paris}}]
		\label{lem: exponentially bounded tails Paris}
		Fix $\frac{1}{3} < \nu < \frac{1}{2}$, set $\epsilon = x^{-\nu}$ and $N_{\pm} := \lfloor (1\pm \epsilon)(hx)^{\frac{1}{\kappa}}\rfloor$. Then, there exist $\alpha, \beta > 0$ such that::
		\begin{equation}
			\frac{\sum_{n = 0}^{N_-} a_nx^n}{F(x)} = O(\exp{(-\al x^{\beta})})
		\end{equation}
		and 
		\begin{equation}
			\frac{\sum_{n = N_+}^{\infty} a_nx^n}{F(x)} = O(\exp{(-\al x^{\beta})}).
		\end{equation}
	\end{nlemma}
	\begin{remark}
		Technically Paris only proves the above result in the special case when $\alpha_r = \beta_r = 1$ for all $r$. As mentioned in Section~6 of op. cit. the method employed there extends to the more general series.
	\end{remark}
	We next use Paris' work to show that the same still holds after scaling the series by any subpolynomially increasing sequence, so that we have:
	\begin{nlemma}
		The series $F(x)$ is exponentially bounded away from $n = (hx)^{\frac{1}{\kappa}}$.
	\end{nlemma}
	\begin{proof}
		As $b_n \leq Bn$ for some $B$, we have, for any $n_1$, $n_2$: \begin{equation}
		\sum_{n = n_1}^{n_2} a_nb_nx^n \leq B\sum_{n = n_1}^{n_2} na_nx^n =  B\sum_{n = n_1}^{n_2} \frac{\Gamma(n+1)}{\Gamma(n)}a_nx^n
		\end{equation}
		Now consider the series $F(x)$ with additional $\alpha_{p+1} = 1$, $a_{p+1} = 1$, $\beta_{q+1} = 1$ and $b_{q+1} = 0$. Call this function $G(x)$. Note that this leaves $\kappa$ and $h$ invariant. Applying Lemma \ref{lem: exponentially bounded tails Paris} to $G(x)$ yields:
		\begin{equation}
		\frac{\sum_{n = 0}^{N_-} \frac{\Gamma(n+1)}{\Gamma(n)}a_nx^n}{G(x)} = O(\exp{(-\widetilde{\al} x^{\widetilde{\beta}})})
		\end{equation}
		and
		\begin{equation}
		\frac{\sum_{n = N_+}^{\infty} \frac{\Gamma(n+1)}{\Gamma(n)}a_nx^n}{G(x)} = O(\exp{(-\widetilde{\al} x^{\widetilde{\beta}})})
		\end{equation}
		for some $\widetilde{\al}, \widetilde{\beta} > 0$. Finally, observe that \begin{equation}
		\frac{G(x)}{F(x)} \leq \frac{\sum_{n = N_-}^{N_+} a_nb_nx^n}{\sum_{n = N_-}^{N_+} a_nx^n} + O(\exp{(-\widetilde{\al} x^{\widetilde{\beta}})}) \leq N_+ + O(\exp{(-\widetilde{\al} x^{\widetilde{\beta}})}) = O((1+\epsilon)(hx)^{\frac{1}{\kappa}})
		\end{equation}
		We thus have: \begin{equation}
		\sum_{n = 0}^{N_-} a_nb_nx^n = O(x^{\frac{1}{\kappa}}\exp{(-\widetilde{\al} x^{\widetilde{\beta}})}) = O(\exp{(-\al' x^{\beta'})})
		\end{equation}
		for any $\al' < \widetilde{\al}$, and $\beta' = \widetilde{\beta}$. Similar reasoning applies to the sum $\sum_{n = N_+}^{\infty} a_nb_nx^n$.
	\end{proof}
	We have thus proven Theorem \ref{thm: hypergeometric series are superpolynmoially peaked}. We will also need:
	\begin{nlemma}
		\label{lem: limit of ratio of series with convergent terms}
		Let $I(x) = \sum_{n \geq 0} a_n x^n$ be a series which converges for all $x$, with $a_n \geq 0$ for all $n$, and such that for all $N$, there exists $n > N$ such that $a_n \neq 0$. Let $b_n$ be a sequence with $\lim_{n \to \infty} b_n = b$, then \begin{equation}
			\lim_{x \to \infty} \frac{\sum_{n \geq 0} a_nb_nx^n}{I(x)} = b.
		\end{equation}
	\end{nlemma}
	\begin{proof}
		Let $\epsilon > 0$. Pick $N$ such that for all $n \geq N$ we have $|b_n - b| < \epsilon$. Then, we have \begin{equation}
			\left\lvert \frac{\sum_{n \geq 0} a_nb_nx^n}{I(x)} - b \right\rvert \leq \frac{\sum_{n \geq 0} a_nx^n |b_n - b|}{I(x)} \leq \frac{\sum_{n = 0}^{N-1} a_nx^n |b_n - b|}{I(x)} + \epsilon \frac{\sum_{n \geq N} a_nx^n}{I(x)}
		\end{equation}
		Consider the ratio $\frac{\sum_{n = 0}^{N-1} a_nb_nx^n}{I(x)}$. By assumption, the numerator is a polynomial of degree at most $N-1$, and the denominator has unbounded degree. Thus, we can choose $x$ large enough such that \begin{equation}
			\left \lvert \frac{\sum_{n = 0}^{N-1} a_nb_nx^n}{I(x)} \right\rvert < \epsilon
		\end{equation}
		Furthermore, as all coefficients $a_n$ are non-negative, we have $\frac{\sum_{n \geq N} a_nx^n}{I(x)} \leq 1$, which finishes the proof.
	\end{proof}
	
	Finally we use the above results to give a proof of Theorem \ref{thm: rephrasing Gamma Conjecture 1}.
	
	\begin{proof}[Proof of Theorem \ref{thm: rephrasing Gamma Conjecture 1}]
		Suppose the assumptions of the theorem hold, with the peak at $n = \frac{T}{r}t$. We then have that \begin{equation}
			\lim_{t \to \infty} \frac{e^{c_1 \log t} \sum_{n \geq 0} J_{rn}t^{rn}}{\langle J(t), pt \rangle} = \lim_{t \to \infty} \frac{\sum_{n \geq 0} e^{c_1 \log(\frac{rn}{T})}J_{rn}t^{rn}}{\sum_{n \geq 0} \langle J_{rn}, pt \rangle t^{rn}}.
		\end{equation}
		Combining this with Lemma \ref{lem: limit of ratio of series with convergent terms} yields Theorem \ref{thm: rephrasing Gamma Conjecture 1}. Note that we can apply Lemma \ref{lem: limit of ratio of series with convergent terms} as any superpolynomially peaked series $I(t) = \sum_{n = 0}^\infty a_nt^n$ is such that for all $N$, there exists $n > N$ such that $a_n \neq 0$.
	\end{proof}
	\section{Gamma Conjecture $1$ for $\mathbb{CP}^N$}
	\label{sect: Gamma Conjecture for projective space}
	We will now show how Theorem \ref{thm: hypergeometric series are superpolynmoially peaked} leads to a straightforward proof of Gamma Conjecture 1 for projective spaces. First recall:
	\begin{nlemma}[{\cite{Giv3}}]
		\label{lem: J function for projective space}
		The J-function of $\mathbb{CP}^N$ is given by:
		\begin{equation}
			J(t) = e^{(N+1)h\log(t)} \sum_{n \geq 0} \frac{t^{(N+1)n}}{\prod_{k=1}^n (h+k)^{N+1}},
		\end{equation}
		where $h \in H^2(\mathbb{CP}^N)$ is the hyperplane class.
	\end{nlemma}
	Rearranging, and expanding in powers of $h$, we find: \begin{equation}
		J_{\mathbb{CP}^N}(t) = e^{(N+1)h\log(t)} \sum_{n \geq 0} \frac{t^{(N+1)n}}{d!^{N+1}}\left(1 - (N+1)\zeta_n(1)h + \dots\right)
	\end{equation}
	In the sum on the right-hand side, the coefficients of powers of $h$ can all be expressed in terms of partial multiple zeta values, defined by: \begin{equation}
		\zeta_d(s_1, \dots, s_k) := \sum_{d \geq n_1 > n_2 > \dots > n_k > 0} \prod_{i = 1}^{k} \frac{1}{n_i^{s_i}}.
	\end{equation}
	In particular, the sequence $\frac{\langle J(t), \alpha \rangle_X}{\langle J(t), pt \rangle}$ is subpolynomially increasing for any $\al \in H^*(\mathbb{CP}^N)$. Moreover, the quantum period is given by $\sum_{n \geq 0} \frac{t^{(N+1)n}}{n!^{N+1}}$, which, by Theorem \ref{thm: hypergeometric series are superpolynmoially peaked} is superpolynomially peaked near $n = t$. Thus, $\mathbb{CP}^N$ satisfies the conditions of Theorem \ref{thm: rephrasing Gamma Conjecture 1} with $r = N+1$ and $T = N+1$. Therefore, we only need to show that: 
	\begin{equation}
		\lim_{n \to \infty} \frac{e^{c_1 \log(n)}J_{rn}}{\langle J_{rn}, pt \rangle} = \Gamma_{\mathbb{CP}^n}.
	\end{equation}
	Writing out the class $\frac{e_{S^1}(\MN^n)^{-1}}{\langle e_{S^1}(\MN^n)^{-1}, pt \rangle}$ and comparing with Lemma \ref{lem: J function for projective space} yields: \begin{equation}
		\frac{e_{S^1}(\MN^n)^{-1}}{\langle e_{S^1}(\MN^n)^{-1}, pt \rangle} \bigg\vert_{u = 1} = \frac{J_{rn}}{\langle J_{rn}, pt \rangle}.
	\end{equation}
	We thus have: \begin{equation}
		\lim_{n \to \infty} e^{c_1\log(n)}\frac{J_{rn}}{\langle J_{rn}, pt \rangle} 
		= \lim_{n \to \infty} e^{c_1\log(n)} \frac{e_{S^1}(\MN^n)^{-1}}{\langle e_{S^1}(\MN^n)^{-1}, pt \rangle}\bigg\vert_{u = 1} = \Gamma_{\mathbb{CP}^n},
	\end{equation}
	which finishes the proof of Gamma Conjecture $1$ for $\mathbb{CP}^n$.
	\section{The $J$-function for twistor bundles}
	\label{sect: computation of the J-function}
	In this section we compute the $J$-function of twistor bundles over hyperbolic manifolds. First we recall: \begin{nthm}[{Theorem A and Corollary B \cite{Evans}}]
		\label{thm: quantum cohomology of FP}
		The small quantum cohomology of the twistor space $\tau: Z \rightarrow M$ over a hyperbolic $6$-manifold $M$ with vanishing Stiefel-Whitney classes is \begin{equation}
			QH^*(Z;\Lambda) \simeq H^*(M;\Lambda)[\al] /(\al^{\star 4} = - 8\al\tau^*\chi+ 8q\al^{\star 2} - 16q^2)
		\end{equation}
		where $\al = c_1(Z)$, $\chi \in H^6*(M;\Lambda)$ is the Euler class of $M$, and $\Lambda = \CC[q]$ is the Novikov ring. Moreover, Let $1 , y_1, \dots, y_m, Vol_M$ be a homogeneous basis for $H^*(M;\CC)$, where $Vol_M$ denotes the volume form. Then a basis for $H^*(Z,\CC)$ is given by: \begin{equation}
			1, \al, \dots, \al^3, y_1, \al y_1, \dots, \al^3y_1, \dots, y_m \al^3, Vol_M, \dots, \al^3 Vol_M,
		\end{equation}
		where all products and powers are classical cup product. The matrix for quantum cup product with $c_1(Z) = \alpha$ on the summand spanned by $y_i, \al y_i, \al^2 y_i, \al^3 y_i$ is given by \begin{equation}
			\begin{pmatrix} 
				0 & 4q & 0 & 0\\
				1 & 0 & 0 & 0\\
				0 & 1 & 0 & 4q\\
				0 & 0 & 1 & 0
			\end{pmatrix}
		\end{equation}
		On the $8$-dimensional summand spanned by $1,\al, \al^2, \al^3, Vol_M, \dots, \al^3 Vol_M$ the action of $c_1 \star$ is given by:
		\begin{equation}
			\begin{pmatrix} 
			0 & 4q & 0 & 0 & 0 & 0 & 0 & 0\\
			1 & 0 & 0 & 0 & 0 & 0 & 0 & 0\\
			0 & 1 & 0 & 4q & 0 & 0 & 0 & 0\\
			0 & 0 & 1 & 0 & 0 & 0 & 0 & 0\\
			0 & 0 & 0 & 0 & 0 & 4q & 0 & 0\\
			0 & 0 & 0 & 8\chi & 1 & 0 & 0 & 0\\
			0 & 0 & 0 & 0 & 0 & 1 & 0 & 4q\\
			0 & 0 & 0 & 0 & 0 & 0 & 1 & 0
		\end{pmatrix}
		\end{equation}
	\end{nthm}
	\begin{remark}
		The statement in \cite[Theorem~A]{Evans} contains a minor typo, by omitting the ``$-$'' sign of $8 \al \tau^*\chi$. The computation in \cite[Section~6]{Evans} shows the above result. The matrix for $c_1 \star$ written down in \cite[Corollary~B]{Evans} is correct, but there is a mistake in the interpretation of the notation. The above matrices are compatible with the description of the quantum cohomology in terms of generators and relations. They imply that the characteristic polynomial of the action of $c_1 \star$ is \begin{equation}
			(\lambda^4 - 8\lambda^2 + 16q^2)^D,
		\end{equation}
		where $D = \dim(H^*(M;\QQ))$. Thus the eigenvalues $\pm 2\sqrt q$ both occur with multiplicity $2D$. We thank Paul Seidel for pointing out the correct $8 \times 8$ matrix.
	\end{remark}
	
	 To compute the J-function, we want to compute the invariants $\langle \psi^{2d+1-k} \alpha^k Vol_M \rangle_d$ and  $\langle \psi^{2d+4-k} \alpha^k \rangle_d$ for $k$ = $0,1,2,3$, where $Vol_M$ is the volume form on $M$. This will be done using the divisor axiom and the topological recursion relation (TRR, see e.g. \cite[Lemma~10.2.2]{CK}). We will prove an algorithm used to compute these invariants and give a single example application of the algorithm. The other invariants will be stated without proof, leaving the computations as an exercise to the reader. We will need the following results proven by Evans:
	\begin{nlemma}[Corollary~7 and Theorem~9 {\cite{Evans}}]
		The invariants $GW_{mA,k} = 0$ for $k \leq 3$ vanish for $m \geq 2$.
	\end{nlemma}
	\begin{nlemma}[{\cite[Corollary~5]{Evans}}]
		The non-zero invariants with $1$ or $2$ marked points are given by:
		\begin{align*}
			\langle Vol_Z \rangle_1 &= 1\\
			\langle Vol_Z, \al \rangle_1 &= 2\\
			\langle \al y_i, \al^3 y_j \rangle_1 &= 16 \int_M y_i \cup y_j
		\end{align*}
	\end{nlemma}
	We will then prove: \begin{nlemma}
		The $1$-point descendent invariants $\langle \psi^* \al^kVol_M^{\delta} \rangle$ with $\delta = 0,1$ and $k = 0,1,2,3$ can be computed from the $1$ and $2$ point invariants described above. Moreover, the descendent invariants $\langle \psi^* \al^k y_i \rangle$ vanish.
	\end{nlemma}
	\begin{proof}
	Group the invariants $\langle \psi^{2d+5-k - j -3(\delta_1 + \delta_2) } \al^kVol_M^{\delta_1}, \al^jVol_M^{\delta_2} \rangle_d$ for $\delta_i \in \{0,1\}$ into sets of 2 by fixing $(\delta_1,k)$ and grouping together $(\delta_2,j)$ as  $\{ \{(0,0), (0,1)\}, \{(0,2), (0,3)\}, \{(1,0), (1,1) \}, \{(1,2), (1,3) \} \} =: Q$. For $q \in Q$ order $q = (q_0, q_1)$ in the obvious lexicographic way. Let $P = \{(\delta_1,k) \in \{0,1\} \times \{0,1,2,3\} \}$. 
	Equip $K$ with the lexicographic partial order, and order $J$ by the order written in the set. We then claim that the degree $d$ invariants with indices $(p,q) \in P \times Q$ can be computed from those with $(d', p',q')$ with either ($d' = d$, $p' = p$ and $q' > q$) or ($d' = d$ and $p' > p$) or ($d' < d$ and $q' = q$). The result will then follow by induction on the degree and the index $(p,q)$.
	
	For ease of writing, we will just write $\psi^*$, as the power $*$ is determined by $d,p,q$ and the degree axiom of descendent invariants. First note that by the divisor axiom we have \begin{equation}
		\langle \psi^{*} \al^kVol_M^{\delta_1}, \al^jVol_M^{\delta_2},\al \rangle_d = \langle \psi^{*-1} \al^{k+1}Vol_M^{\delta_1}, \al^jVol_M^{\delta_2} \rangle_d + 2d\langle \psi^{*} \al^kVol_M^{\delta_1}, \al^jVol_M^{\delta_2} \rangle_d.
	\end{equation}
	Here $\al^{k+1}$ should be read as $8\chi Vol_M$ when $k=3$, in which case there is no contribution from the first term on the right if $\delta_1 = 1$. This introduces the invariants with ($d' = d$ and $p' > p$). We can also apply the topological recursion relation (see e.g. \cite[Lemma~10.2.2]{CK}) to obtain: \begin{multline}
		\langle \psi^{*} \al^kVol_M^{\delta_1}, \al^jVol_M^{\delta_2},\al \rangle_d \\=  \sum_{m,n}\langle \psi^{*-1} \al^{k}Vol_M^{\delta_1}, \eta_m \rangle_d g^{mn} \langle \eta_n, \al^jVol_M^{\delta_2},\al \rangle_0 + \langle \psi^{*-1} \al^kVol_M^{\delta_1}, \eta_m \rangle_{d-1} g^{mn} \langle \eta_n, \al^jVol_M^{\delta_2}, \al \rangle_1,
	\end{multline}
	noting that there are no non-zero $3$-point Gromov-Witten invariants without $\psi$-classes in degree $d>1$. We first focus on the $d=0$ contributions. Separate into the cases for different $\delta_2$.
	\\
	
	$\delta_2 = 1$: the only non-zero invariant $\langle \eta_n, \al^jVol_M,\al \rangle_0$ is for $\eta_n = \al^{2-j}$, we thus have $\eta_m = \al^{j+1}Vol_M$, which indeed gives us an invariant with index $q' > q$.
	\\
	
	$\delta_2 = 0$: now we have contributions from either $\eta_n = \al^{5-j}$, in which case we get $\eta_m = \al^{j+1}$ or $\al^{j-2}Vol_M$, or we have $\eta_n = Vol_M \al^{2-j}$, which gives $\eta_m = \al^{j+1}$, or if $j = 2$ additionally $\eta_m = Vol_M$. All of these give us invariants with index $q' > q$. We next look at the contributions from degree $d = 1$. The only non-zero invariants $\langle \eta_n, \al^jVol_M^{\delta_2}, \al \rangle_1$ are for:
	
	$(\delta_2,j) = (0,1)$: we must have $\eta_n = \al^3Vol_M$, so that $\eta_m = 1$.\\
	
	$(\delta_2,j) = (0,3)$: we must have $\eta_n = \al Vol_M$, so that $\eta_m = \al^2$.\\
	
	$(\delta_2,j) = (1,1)$: we must have $\eta_n = \al^3$, so that $\eta_m = Vol_M$.\\
	
	$(\delta_2,j) = (1,3)$: we must have $\eta_n = \al^3Vol_M$, so that $\eta_m = 1$.\\
	
	All the invariants $\langle \psi^{*-1} \al^kVol_M^{\delta_1}, \eta_m \rangle_{d-1}$ thus satisfy $q' = q$, which proves the first statement.
	
	For the second statement, note that we can apply a similar induction process as before. In this case, all the invariants of the form $\langle \al^j y_i, \al^3\rangle$ vanish, which then implies that the invariants $\langle \al^j y_i \rangle$ vanish as well.
	\end{proof}
	The magic property of the above inductive process is that it expresses the invariant with index $(d,p,q_0 \in q)$ in terms of invariants which have already been computed using induction, and an invariant with index $(d,p,q_1 \in q)$. Applying this procedure again, we express the original invariant in terms of an invariant with index $(d-1,p,q_0 \in q)$. This thus gives an equation for the invariant $(d,p,q_0 \in q)$. We can actually \emph{solve} the associated recursion relation to obtain explicit expressions for all the invariants involved. The relevant invariants for the J-function are:
	\begin{nthm}
		The 1-point descendent invariants are given by:
		\begin{align*}
			\langle \psi^{2d-2} \al^3Vol_M \rangle_d &= \frac{8}{d!^2}\\
			\langle \psi^{2d-1} \al^2Vol_M \rangle_d &= -\frac{8}{d!^2}\sum_{k=1}^d \frac{1}{k}\\
			\langle \psi^{2d} \al Vol_M \rangle_d &= \frac{8}{d!^2}\sum_{\substack{k_1 \leq k_2 \leq d}} \frac{1}{k_1k_2}\\
			\langle \psi^{2d+1}  Vol_M \rangle_d &= -\frac{8}{d!^2}\sum_{\substack{k_1 \leq k_2 \leq k_3 \leq d}} \frac{1}{k_1k_2k_3}\\
			\langle \psi^{2d+1}  \al^3 \rangle_d &= -\frac{8\chi}{d!^2}\sum_{\substack{k_1 \leq k_2 \leq d}} \frac{1}{k_1k_2^2} +\frac{64\chi}{d!^2}\sum_{\substack{k_1 \leq k_2 \leq k_3 \leq d}} \frac{1}{k_1k_2k_3},
		\end{align*}
	and the longer expressions
	\begin{equation*}
		\langle \psi^{2d+2}  \al^2 \rangle_d = \frac{\chi}{d!^2} \left( \sum_{\substack{k_1 \leq k_2 \leq k_3 \leq d}} \left( \frac{8}{k_1k_2^2k_3} + \frac{16}{k_1k_2k_3^2} \right)  - \sum_{\substack{k_1 \leq \dots \leq k_4 \leq d}} \frac{64}{k_1k_2k_3k_4} \right),
	\end{equation*}
\begin{equation*}
	\langle \psi^{2d+3}  \al \rangle_d = \frac{\chi}{d!^2} \left( \sum_{\substack{k_1 \leq k_2 \leq k_3 \leq d}} \frac{4}{k_1k_2^2k_3^2} - \sum_{\substack{k_1 \leq \dots \leq k_4 \leq d}} \left( \frac{24}{k_1k_2k_3k_4^2} +\frac{16}{k_1k_2k_3^2k_4} + \frac{8}{k_1k_2^2k_3k_4} \right) + \sum_{\substack{k_1 \leq \dots \leq k_5 \leq d}} \frac{64}{k_1\dots k_5}\right),
\end{equation*}
and finally: \begin{multline*}
	\langle \psi^{2d+4}  1 \rangle_d = \frac{\chi}{d!^2} \bigg( \sum_{\substack{k_1 \leq \dots \leq k_4 \leq d}} \left( \frac{-8}{k_1k_2k_3^2k_4^2} +\frac{-4}{k_1k_2^2k_3k_4^2} + \frac{-4}{k_1k_2^2k_3^2k_4} \right)\\
	+ \sum_{\substack{k_1 \leq \dots \leq k_5 \leq d}} \left(\frac{8}{k_1k_2^2k_3k_4 k_5} + \frac{16}{k_1k_2k_3^2k_4 k_5} + \frac{24}{k_1k_2k_3k_4^2 k_5} + \frac{32}{k_1k_2k_3k_4 k_5^2} \right) - \sum_{\substack{k_1 \leq \dots \leq k_6 \leq d}} \frac{64}{k_1\dots k_6}\bigg).
\end{multline*}
	\end{nthm}
 	Let \begin{equation}
 		S_d(n_1, \dots, n_i) := \sum_{\substack{d \geq k_1 \geq \dots \geq k_i \geq 1}} \frac{1}{k_1^{n_1} \dots k_i^{n_i}}.
 	\end{equation}
 	We then find: \begin{ncor}
 		The coefficients of the J-function for the twistor spaces are given by: \begin{multline}
 			\frac{J_{2d}}{\langle J_{2d}, pt \rangle} = 1 - S_d(1) \alpha + S_d(1,1) \al^2 - S_d(1,1,1) \al^3 - S_d(2,1)\tau^*\chi - \left( 8 S_d(1,1,1,1) - S_d(1,2,1) - 2 S_d(2,1,1) \right)\alpha \tau^*\chi\\ - \left( 3S_d(2,1,1,1) + 2 S_d(1,2,1,1) + S_d(1,1,2,1) - \frac{1}{2}S_d(2,2,1) - 8 S_d(1,1,1,1,1)  \right)\al^2 \tau^* \chi\\
 			- \bigg( S_d(\{2\}^2, \{1\}^2) + \frac{1}{2}S_d(\{2,1\}^2) +  \frac{1}{2}S_d(1,\{2\}^2,1) - S_d(\{1\}^3,2,1) - 2S_d(\{1\}^2,2,\{1\}^2)\\ - 3 S_d(1,2,\{1\}^3) - 4 S_d(2,\{1\}^4) + 8S_d(\{1\}^6)  \bigg)\al^3 \tau^* \chi 
 			\end{multline}
 	\end{ncor}
 	We then conclude:
 	\begin{ncor}
 		\label{cor: J function is exponentially peaked}
 		The $J$-function of the twistor bundle over a hyperbolic manifold with vanishing Stiefel-Whitney classes satisfies the assumptions of Theorem \ref{thm: rephrasing Gamma Conjecture 1}.
 	\end{ncor}
 	\begin{proof}
 		The quantum period is given by $\langle J_X, Vol_X \rangle_X = \sum_{n \geq 0 } \frac{t^{2n}}{n!^2}$, which is superpolynomially peaked near $n = t$ by Proposition \ref{thm: hypergeometric series are superpolynmoially peaked}. Next, observe that all symmetric sums $S_n(a_1, \dots, a_k)$ are subpolynomially increasing for any $a_1, \dots, a_k$.
 	\end{proof}
	We will illustrate the induction process for the first three descendent invariants. The others follow by similar computations.
	First note that by the divisor axiom we have \begin{equation}
	\langle \psi^{2d-2} Vol_Z, \al, \al \rangle_d = 4d^2 \langle \psi^{2d-2} Vol_Z \rangle_d.
\end{equation}
	Next use the Topological recursion relation to obtain: \begin{multline}
		\label{eq: TRR for fundamental term}
		\langle \psi^{2d-2} Vol_Z, \al, \al \rangle_d =  \langle \psi^{2d-3} Vol_Z, \al^2 \rangle_d \cdot \langle \al Vol_M, \al, \al \rangle_0 \cdot \frac{1}{8}
		+ \langle \psi^{2d-3} Vol_Z, 1 \rangle_{d-1} \cdot \langle Vol_Z, \al, \al \rangle_1
	\end{multline}
	There are no additional terms by the following lemma: 
	\begin{nlemma}
		 $\langle \psi^{2d-3} Vol_Z, \al^2 \rangle_d = 0 = \langle \psi^{2d-4} Vol_Z, \al^3 \rangle_d = 0$ for all $d$.
	\end{nlemma}
	\begin{proof}
		We first consider $\langle \psi^{2d-4} Vol_Z, \al^3 \rangle_d$. By using the divisor axiom and TRR, we have: \begin{align}
			2d\langle \psi^{2d-4} Vol_Z, \al^3 \rangle_d &= \langle \psi^{2d-4} Vol_Z, \al^3,\al \rangle_d\\
			&= \langle \psi^{2d-5} Vol_Z, \_ \rangle_d \langle \_, \al^3, \al \rangle_0 + \langle \psi^{2d-5} Vol_Z, \al^2 \rangle_{d-1}\langle \al Vol_M, \al^3,\al \rangle_{1}\cdot \frac{1}{8}\\
			&= 4\langle \psi^{2d-5} Vol_Z, \al^2 \rangle_{d-1}
		\end{align}
		As $\al \cup \al^3 = 0$, and the only non-zero degree $1$ invariant with inputs $\al^3$ and $\al$ is the one described.
		Next, we have
		\begin{align}
		2d \langle \psi^{2d-3} Vol_Z, \al^2 \rangle_d &= \langle \psi^{2d-3} Vol_Z, \al^2, \al \rangle_d\\
			&= \langle \psi^{2d-4} Vol_Z, \al^3 \rangle_d\langle Vol_M, \al^2,\al \rangle_0\cdot\frac{1}{8} + \langle \psi^{2d-4} Vol_Z, \_ \rangle_{d-1} \langle \_, \al^2,\al \rangle_1\\
			&= \langle \psi^{2d-4} Vol_Z, \al^3 \rangle_d
		\end{align}
	The final equality follows as by \cite[Corollary~5]{Evans}, the degree $1$, $2$ point invariants with input $\al^2$ vanish, thus, by the divisor axiom $\langle \_, \al^2,\al \rangle_1 = 0$. We can thus conclude the result by induction, noting that the cases $d = 0,1$ hold by definition.
	\end{proof}
	We can then conclude:
	\begin{ncor}
		\label{cor: descendents for fundamental term Jfun}
		$\langle \psi^{2d-2} Vol_Z \rangle_d = \frac{1}{d!^2}$
	\end{ncor}
	\begin{proof}
		Combining equation \eqref{eq: TRR for fundamental term} with the previous lemma, we find $\langle \psi^{2d-2} Vol_Z \rangle_d = \frac{1}{d^2}\langle \psi^{2d-4} Vol_Z \rangle_{d-1}$ for $d \geq 2$. Moreover, for $d =1$, we have $\langle Vol_Z \rangle_1 = 1$ by \cite[Corollary~5]{Evans}.
	\end{proof}
To compute the next term in the J-function, we need to find $\langle \psi^{2d-1} \al^2 Vol_M \rangle_d$. To this end, first observe that the divisor equation now has additional terms, as we are working with descendent invariants: \begin{equation}
	\label{eq: divisor for first term Jfun}
	\langle \psi^{2d-1} \al^2Vol_M,\al,\al \rangle_d = 32d \langle \psi^{2d-2} Vol_Z \rangle_d + 4d^2 \langle \psi^{2d-1} \al^2Vol_M \rangle_d.
\end{equation}
We will then compute $\langle \psi^{2d-1} \al^2Vol_M,\al,\al \rangle_d$ using the TRR. We find: \begin{align}
	\langle \psi^{2d-1} \al^2Vol_M,\al,\al \rangle_d &= \langle \psi^{2d-2} \al^2Vol_M,\al^2 \rangle_d \langle \al Vol_M,\al,\al \rangle_0 \cdot \frac{1}{8} + \langle \psi^{2d-2} \al^2Vol_M,1 \rangle_{d-1} \langle Vol_Z,\al,\al \rangle_1\\
	\label{eq: TRR for first term Jfun}
	&=  \langle \psi^{2d-2} \al^2Vol_M,\al^2 \rangle_d + 4  \langle \psi^{2d-3} \al^2Vol_M \rangle_{d-1}.
\end{align}
Similar reasoning as before, taking care to use the divisor axiom correctly, again yields:
\begin{nlemma}
	\label{lem: al2 and al3 vanish for al2volM}
	$\langle \psi^{2d-2} \al^2Vol_M,\al^2 \rangle_d = 0= \langle \psi^{2d-3} \al^2Vol_M,\al^3 \rangle_d$ for all $d$.
\end{nlemma}
We thus obtain: \begin{ncor}
	 $\langle \psi^{2d-1} \al^2Vol_M \rangle_d = -\frac{8}{d!^2}\sum_{k=1}^d \frac{1}{k}$.
\end{ncor}
\begin{proof}
	Combining Equations \eqref{eq: divisor for first term Jfun} and \eqref{eq: TRR for first term Jfun} yields \begin{equation}
		\langle \psi^{2d-1} \al^2Vol_M \rangle_d = \frac{1}{d^2}\langle \psi^{2d-3} \al^2Vol_M \rangle_{d-1} - \frac{8}{d}\langle \psi^{2d-2} Vol_Z \rangle_d.
	\end{equation}
	Now use corollary \ref{cor: descendents for fundamental term Jfun} and solve the recursion relation.
	\end{proof}
Next we compute $\langle \psi^{2d} \al Vol_M \rangle_d$. We first find:
\begin{nlemma}
	We have $\langle \psi^{2d-1} \al Vol_M,\al^2 \rangle_d = \frac{8}{d!^2}$ and $\langle \psi^{2d-2} \al Vol_M,\al^3 \rangle_d = \frac{16d}{d!^2}$.
\end{nlemma}
\begin{proof}
	The difference with the previous lemmas of this form is that we now can't apply the TRR when $d=1$ for the second invariant, as there are no $\psi$-classes. We first find: \begin{equation}
	\langle \al Vol_M,\al^3 \rangle_1 = 16
\end{equation}
	from \cite[Corollary~5]{Evans}. For $d \geq 1$ then observe that:
	\begin{align}
		\langle \psi^{2d-2} \al Vol_M,\al^3 \rangle_d = \frac{1}{2d}\langle \psi^{2d-2} \al Vol_M,\al^3,\al \rangle_d
	\end{align}
	as the extra contributions from the divisor axiom vanish by \ref{lem: al2 and al3 vanish for al2volM}. Then use the TRR to obtain \begin{equation}
		\langle \psi^{2d-2} \al Vol_M,\al^3,\al \rangle_d = 4\langle \psi^{2d-3} \al Vol_M,\al^2\rangle_{d-1}.
	\end{equation}
	Finally we use the TRR again and the divisor equation to obtain \begin{equation}
	\langle \psi^{2d-1} \al Vol_M,\al^2 \rangle_d = \frac{1}{2d} \langle \psi^{2d-2} \al Vol_M,\al^3\rangle_{d}.
	\end{equation}
	The result then follows by induction.
\end{proof}
\begin{ncor}
	We have \begin{equation}
	\langle \psi^{2d} \al Vol_M \rangle_d = \frac{8}{d!^2}\sum_{k=1}^{d} \frac{1}{k}\sum_{j=1}^{k} \frac{1}{j}.
	\end{equation}
\end{ncor}
\begin{proof}
	The divisor equation yields:
	\begin{equation}
		\langle \psi^{2d} \al Vol_M,\al,\al \rangle_d = \langle \psi^{2d-2} \al^3 Vol_M \rangle_d + 4d\langle \psi^{2d-1} \al Vol_M \rangle_d + 4d^2\langle \psi^{2d} \al Vol_M \rangle_d.
	\end{equation}
	On the other hand, the TRR gives: \begin{equation}
		\langle \psi^{2d} \al Vol_M,\al,\al \rangle_d = \langle \psi^{2d-1} \al Vol_M,\al^2 \rangle_d + 4 \langle \psi^{2d-2} \al Vol_M \rangle_{d-1}.
	\end{equation}
	The result then follows by induction.
\end{proof}

\section{Proof of Gamma Conjecture $1$ for the twistor bundles}
\label{sect: Gamma Conjecture FP}
In this section we will finish the proof of Theorem \ref{thm: FP satisfies Gamma 1}. By Corollary \ref{cor: J function is exponentially peaked}, Theorem \ref{thm: rephrasing Gamma Conjecture 1} applies, so that we only need to show: 
\begin{equation}
	\lim_{n \to \infty} e^{c_1\log(n)}\frac{J_{2n}}{\langle J_{2n}, pt \rangle} = \Gamma_Z.
\end{equation} 
Unlike for projective spaces, we don't have that \begin{equation}
	\frac{J_{2n}}{\langle J_{2n}, pt \rangle} = \frac{e_{S^1}(\MN^n)^{-1}}{\langle e_{S^1}(\MN^n)^{-1}, pt \rangle},
\end{equation}
which would then immediately show the Gamma Conjecture. Thus, we compute the limit directly. We do this separately for each cohomological degree, to this end, for each even cohomological degree $2i$, let $P_{i,n}$ be the sequence $\frac{J_{2n}}{\langle J_{2n}, pt \rangle}\big\vert_{H^{2i}(Z)}$. Recall that the Gamma class is given by: \begin{equation}
	\Gamma_Z = \exp \left( -C_{eu}ch_1 + \sum_{k\geq 2} (-1)^k (k-1)!\zeta(k) ch_k(Z)\right).
\end{equation}
Expand the Gamma class for $i\geq 2$ as \begin{equation}
	\Gamma_Z|_{H^{2i}(Z)} = (-1)^k\zeta(k)ch_k + h_i(-C_{eu}ch_1, \zeta(2)ch_2, \dots, (-1)^{i-1}(i-2)!\zeta(i-1)ch_{i-1}),
\end{equation}
where $h_i$ is some quasi-homogeneous polynomial of degree $i$ in $i-1$ variables (which we will compute later). For completeness, set $h_0 = h_i = 0$. Then, define \begin{equation}
	R_{i,n} := \left(e^{c_1\log(n)}\frac{J_{2n}}{\langle J_{2n}, pt \rangle}\right)\bigg\vert_{H^{2i}(Z)} - h_i(R_{1,n}, \dots, R_{i-1,n}) = \sum_{j = 0}^{i} \left(\frac{\log(n)^jc_1^j}{j!}P_{i-j,n}\right) - h_i(R_{1,n}, \dots, R_{i-1,n}).
\end{equation}
We will then show: \begin{nprop}
	\label{prop: final limits J function}
	For each $i$, we have \begin{equation}
		\lim_{n \to \infty} R_{i,n} = (-1)^k (k-1)!\zeta(k) ch_k(Z).
	\end{equation}
\end{nprop}
Once this has been shown, the definition of the $h_i$ then immediately shows the required result: \begin{equation}
	\lim_{n \to \infty} e^{c_1\log(n)}\frac{J_{2n}}{\langle J_{2n}, pt \rangle} = \Gamma_Z
\end{equation}
First, a direct computation shows: \begin{nlemma}
	The polynomials $h_i(x_1, \dots, x_i)$ are given by:
	\begin{align}
		h_2 & = \frac{1}{2}x_1^2\\
		h_3 &= \frac{1}{6}x_1^3 + x_1x_2\\
		h_4 &= \frac{1}{24}x_1^4 + \frac{1}{2}x_1^2x_2 + \frac{1}{2}x_2^2 + x_1x_3\\
		h_5 &= \frac{1}{120}x_1^5 + \frac{1}{6}x_1^3x_2 + \frac{1}{2}x_1x_2^2 + \frac{1}{2}x_1^2x_3 + x_1x_4 + x_2x_3\\
		h_6 &= \frac{1}{720}x_1^6 + \frac{1}{24}x_1^4x_2 + \frac{1}{4}x_1^2x_2^2 + \frac{1}{6}x_2^3 + \frac{1}{6}x_1^3x_3 + \frac{1}{2}x_1^2x_4 + x_1x_5 + \frac{1}{2}x_3^2 + x_1x_2x_3 + x_2x_4
	\end{align}
\end{nlemma}
We then compute the Gamma class of the twistor bundles. Evans \cite[Section~6]{Evans} computes the Chern classes of the horizontal distribution of the twistor bundles. A short computation then shows: \begin{nlemma}
	The Chern character of the twistor bundles is given by:
	\begin{align*}
		ch_1(Z) &= \al  &ch_4(Z) &= \frac{\al\tau^*\chi}{6}\\
		ch_2(Z) &= \frac{\al^2}{2} &ch_5(Z) &= \frac{7 \al^2 \tau^*\chi}{120}\\
		ch_3(Z) &= \frac{\al^3}{6} + \tau^*\chi &ch_6(Z) &= \frac{7 \al^3 \tau^*\chi}{720}
	\end{align*}
\end{nlemma}
We are now ready to prove Proposition \ref{prop: final limits J function}.
\begin{proof}[Proof of Proposition \ref{prop: final limits J function}]
	For degree $0$ there is nothing to prove. For degree $2$ we have:
	\begin{equation}
		R_{1,n} = P_{1,n} + c_1\log(n)P_{0,n} = \left(\log(n) - S_n(1)\right)\al \xrightarrow{n\to\infty} -C_{eu}\al.
	\end{equation}
	as required. For degree 4, we find that
		\begin{equation}
		R_{2,n} = P_{2,n} - \frac{1}{2}P_{1,n}^2 =  \left(S_n(1,1) - \frac{1}{2}S_n(1)^2\right)\al^2.
	\end{equation}
	Now we use the fact that we can expand $S_n(I)$, and products thereof, as sums of (partial) multiple zeta values (with different indices). In this case one has: \begin{align}
		S_n(1,1) &= \zeta_n(1,1) + \zeta_n(2)\\
		S_n(1)^2 &= 2\zeta_n(1,1) + \zeta_n(2).
	\end{align}
	We thus find:
	\begin{equation}
		R_{2,n} = \left(S_n(1,1) - \frac{1}{2}S_n(1)^2\right)\al^2 \xrightarrow{n\to\infty} \frac{1}{2}\zeta(2)\al^2 = \zeta(2)ch_2(Z),
	\end{equation}
	as required. Next up is degree $6$. Here we have: \begin{equation}
		R_{3,n} = P_{3,n} + \frac{1}{3}P_{1,n}^3 - P_{1,n}P_{2,n} = \left(-S_n(1,1,1) - \frac{1}{3}S_n(1)^3 +S_n(1)S_n(1,1)\right)\al^3 - S_n(2,1)\tau^*\chi
	\end{equation}
	 Expanding the $S_n$ into multiple zeta function, we find: \begin{equation}
	 	R_{3,n} = -\frac{1}{3}\zeta_n(3)\al^3 - \left(\zeta_n(3) + \zeta_n(2,1) \right)\tau^*\chi.
	 \end{equation}
	 Now, combine this with Euler's computation that $\zeta(2,1) = \zeta(3)$ to find: \begin{equation}
	 	\lim_{n \to \infty} R_{3,n} = -\frac{1}{3}\zeta(3)\al^3 - 2\zeta(3)\tau^\chi = -2\zeta(3)ch_3(Z).
	 \end{equation}
	 For degree $8$, the computations become lengthy by hand, so the computational package SAGE was used (code can be found on the author's website). This shows: \begin{equation}
	 	R_{4,n} = P_{4,n} - \frac{1}{2}P_{2,n}^2 + P_{1,n}^2P_{2,n} - P_{1,n}P_{3,n} - \frac{1}{4}P_{1,n}^4,
	 \end{equation}
	 Using SAGE to expand $R_{4,n}$ in multiple zeta values, we find: \begin{equation}
	 	R_{4,n} = \left(\zeta_n(2,2) + \zeta_n(3,1)\right)\al\tau^*\chi.
	 \end{equation}
	 Again, Euler's sum theorem for multiple zeta values shows that $\zeta(2,2) + \zeta(3,1) = \zeta(4)$, so that \begin{equation}
	 	\lim_{n \to \infty} R_{4,n} = \zeta(4)\al\tau^*\chi =  6\zeta(4)ch_4(Z).
	 \end{equation}
	 For cohomological degree $10$ we have: \begin{equation}
	 	R_{5,n} = \frac{1}{5} P_{1,n}^{5} - P_{1,n}^{3} P_{2,n} + P_{1,n} P_{2,n}^{2} + P_{1,n}^{2} P_{3,n} - P_{2,n} P_{3,n} - P_{1,n} P_{4,n} + P_{5,n},
	 \end{equation}
	 Expanding $R_{5,n}$ in multiple zeta values, we find: \begin{equation}
	 	R_{5,n} = \left(\frac{1}{2}\zeta_n(2,2,1) - \frac{1}{2}\zeta_n(2,3) - \zeta_n(3,2) - \frac{1}{2}\zeta_n(4,1) - \frac{9}{10}\zeta_n(5)\right)\al^2\tau^*\chi.
	 \end{equation}
	 Now use the cyclic sum theorem  (see \cite{Hoff}) to rewrite $\zeta(2,2,1) = \zeta(3,2)$. Combining this with another version of Euler's sum theorem, which states \begin{equation}
	 	\zeta(2,3) + \zeta(3,2) +\zeta(4,1) = \zeta(5),
	 \end{equation}
	 yields that: \begin{equation}
	 	\lim_{n \to \infty} R_{5,n} = -\frac{7}{5} \al^2\tau^*\chi = -24\zeta(5)ch_5(Z),
	 \end{equation}
	 as required. Finally for degree $12$, we have: \begin{equation}
	 	R_{6,n} = -\frac{1}{6} P_{1,n}^{6} + P_{1,n}^{4} P_{2,n} - \frac{3}{2} P_{1,n}^{2} P_{2,n}^{2} - P_{1,n}^{3} P_{3,n} + \frac{1}{3} P_{2,n}^{3} + 2 P_{1,n} P_{2,n} P_{3,n} + P_{1,n}^{2} P_{4,n} - \frac{1}{2} P_{3,n}^{2} - P_{2,n} P_{4,n} - P_{1,n} P_{5,n} + P_{6,n},
	 \end{equation}
	 Expanding $R_{6,n}$ in multiple zeta values, a direct computation shows: \begin{equation}
	 	R_{6,n} = \left(\frac{7}{6}\zeta_n(6) + \frac{1}{2} \left(\zeta_n(3,3) + \zeta_n(4,2) - \zeta_n(2,2,2) - \zeta_n(2,3,1) - \zeta_n(3,2,1) \right)  \right)\al^3\tau^*\chi.
	 \end{equation}
	 Applying the cyclic sum theorem shows that \begin{equation}
	 	\zeta(3,3) + \zeta(4,2) = \zeta(2,2,2) + \zeta(2,3,1) + \zeta(3,2,1).
	 \end{equation}
	 We thus have: \begin{equation}
	 	\lim_{n \to \infty}R_{6,n} =  \frac{7}{6}\zeta(6)\al^3\tau^*\chi = 120 \zeta(6)ch_6(Z),
	 \end{equation}
	 which finishes the proof.
\end{proof} 
\section{Unramified exponential type}
\label{sect: FP has unram}
Let $\E$ be a finite-dimensional $\CC((u))$-vector space, equipped with a connection $\nabla_{\frac{d}{du}}: \E \rightarrow \E$. First recall:
\begin{defi}
	The connection $\nabla_{\frac{d}{du}}$ is said to be regular singular if there exists a $\CC[[u]]$-submodule $\Lambda \subset \E$ of maximal rank such that $u\nabla_{\frac{d}{du}}: \Lambda \rightarrow \Lambda$.
\end{defi}
Alternatively, a connection is regular singular if there exists a basis such that the connection matrix for $\nabla_{\frac{d}{du}}$ has a pole of order $1$. Next, for $\phi \in \CC[u^{-1}]$, let $\E^\phi$ be the rank $1$ $\CC((u))$-vector space with connection $\frac{d}{du} + \frac{d\phi}{du}$. The Hukuhuhara-Level-Turrittin decomposition states:
\begin{nthm}[{see \cite{Huk}, \cite{Lev}, \cite{Tur}}]
	Let $\nabla: \E \rightarrow \E$ be any connection, then there exists an integer $r$, such that we have a decomposition \begin{equation}
		\left(\E \otimes_{\CC((u))} \CC((s)), \nabla\right) = \bigoplus_{\phi \in \Phi} \E^\phi \otimes R_\phi,
	\end{equation}
	where the map $\CC((u)) \rightarrow \CC((s))$ is given by $u \mapsto s^r$, $\Phi \subset \CC[s^{-1}]$ is a finite set, and $R_\phi$ has a regular singular connection for all $\phi$. 
\end{nthm}
\begin{defi}
	The connection $\nabla_{\frac{d}{du}}$ is said to be of unramified exponential type is such a decomposition exists for $r = 1$. 
\end{defi}
\begin{conjecture}[{\cite[Conjecture~3.4]{KKP}}]
	Let $X$ be a K\"ahler manifold, then the associated quantum connection $\nabla_{\frac{d}{du}}$ is of unramified exponential type.
\end{conjecture}
\begin{remark}
	In \cite{KKP}, the conjecture is stated for all symplectic manifolds. However, in Remark 3.5.ii they caveat this with the statement that there might be non-K\"ahler manifolds for which the conjecture doesn't hold.
\end{remark}
In this section we prove this conjecture for twistor bundles of hyperbolic manifolds with vanishing Stiefel-Whitney classes (Theorem \ref{thm: FP has unramified exponential type}). First we need some general theory on $\MD$-modules.
\subsection{The irregularity number}
	Let $\KK$ be an algebraically complete field and let $\kk = \KK((u))$ and $R = \KK[[u]]$. Let $\mathcal{D} = \kk[\frac{d}{du}]$ be the Weyl-algebra associated to $\kk$. Here $\frac{d}{du}$ is the usual differential, satisfying $[\frac{d}{du},u] = 1$. Let $M$ be a finite dimensional $\kk$-vector space. A maximal rank $R$-module $\Lambda \subset M$ is called a lattice. Let $\nabla := \nabla_{d/du}:M \rightarrow M$ be a differential operator, i.e. a $\KK$-linear map such that for $f \in \kk$ \begin{equation}
	\nabla(f\al) = \frac{df}{du}\al + f \nabla(\al) \text{ for all } \al \in M.
\end{equation} 
Such $\nabla$ makes $M$ into a $\mathcal{D}$-module. \begin{defi}
	Given a lattice $\Lambda \subset M$, there exists an integer $g$ such that $u^g\nabla \Lambda \subset \Lambda$. The minimal such $g$ is called a \emph{order} of the lattice. The order of $(M,\Lambda)$ is defined to be the minimal order of any lattice. A lattice is called \emph{logarithmic} if it has order $\leq 1$. Then, $(M,\nabla)$ is called a \emph{regular singular} $\MD$-module if it admits a logarithmic lattice $\Lambda \subset M$. Otherwise it is called \emph{irregular}.
\end{defi}
For simplicity, let $D := u\nabla$.
\begin{defi}[cyclic vector]
	An element $e \in M$ is called a \emph{cyclic vector} if $e, D e, D^2 e, \dots, D^{dim M - 1}e$ span $M$ as a $\kk$-vector space. Let $n = dim M$, then we obtain an equation  \begin{equation}
		D^{n}e + a_{n-1}D^{n-1}e + \dots + a_0e = 0 \text{ for some } a_i \in \kk.
	\end{equation}
	Set $\partial = u\frac{d}{du} \in \mathcal{D}$. Call the operator $L = \partial^{n} + a_{n-1}\partial^{n-1} + \dots + a_0 \in \mathcal{D}$ the \emph{differential operator associated to} $(M,\nabla)$.
\end{defi}
\begin{nlemma}[See e.g. \cite{Katz}]
	Any $\mathcal{D}$-module admits a cyclic vector.
\end{nlemma}
\begin{defi}
	Let \begin{equation}
		L =  \partial^{n} + a_{n-1}\partial^{n-1} + \dots + a_0 \MD
	\end{equation}
	be a monic differential operator. Define the \emph{irregularity number} \begin{equation}
		Irr(L) = max_i -\nu(a_i),
	\end{equation} where $\nu(f)$ is the order of vanishing of $f \in \kk$ at $u=0$. By construction, $\nu(a_n) = 0$.
\end{defi}
\begin{eg}
	Let $L = \partial^2 + u^{-2}\partial + u^3$. Then $Irr(L) = 2$.
\end{eg}
\begin{defi}
	Given a $\MD$-module $(M,\nabla)$, define \begin{equation}Irr(M,\nabla) := Irr(L)\end{equation} for any differential operator $L$ associated to $(M,\nabla)$. Whenever it is clear which connection $\nabla$ we are using, we will just write $Irr(M)$.
\end{defi}
The above definition requires that the irregularity number is the same for all differential operators associated to $(M,\nabla)$. This was shown by Malgrange \cite{Mal72}. The proof goes via the following Lemma:

\begin{nlemma}[{\cite[Section~5]{Mal72}}]
	\label{lem: alternative defi of irregularity}
	Let $\Lambda, \Lambda' \subset M$ be two lattices such that $\Lambda \subset \Lambda'$ and $D(\Lambda) \subset \Lambda'$. We then have:
	\begin{equation}Irr(M,\nabla) = \chi(D;\Lambda,\Lambda') + Dim_\KK \Lambda'/\Lambda.\end{equation}
	Here $\chi(D;\Lambda,\Lambda') := dim_\KK ker D - dim_\KK coker D$ is the index of $D: \Lambda \rightarrow \Lambda'$. Moreover, the number on the right hand side is independent of the choice of lattices $\Lambda,\Lambda'$.
\end{nlemma}
As both $\chi(D;\Lambda,\Lambda')$ and $Dim_\KK \Lambda'/\Lambda$ are additive under direct sums, we obtain:
\begin{nlemma}
	\label{lem: irregularity number is additive}
	Let $(M_1,\nabla_1)$ and $(M_2,\nabla_2)$ be $\MD$-modules, then $Irr(M_1 \oplus M_2) = Irr(M_1) \oplus Irr(M_2)$.
\end{nlemma}
We will need the following two propositions:
\begin{nprop}[{\cite[Proposition~5.5]{Mal72}}]
	\label{prop: regular singular iff 0 irregularity}
	A $\MD$-module $(M,\nabla)$ is regular singular if and only if $Irr(M) = 0$.
\end{nprop}
\begin{nprop}[{\cite[Proposition~5.6]{Mal72}}]
	\label{prop: invertible polar term gives maximum irregularity}
	Let $(M,\nabla)$ be an $n$-dimensional $\MD$-module. Choose any basis for $M$, and write $D = u\frac{d}{du} + \sum_{i\geq -g}^{\infty}A_iu^i$ for some $g \geq 0$ and matrices $A_i: \KK^m \rightarrow \KK^m$. Then $Irr(M) \leq gm$, with equality if and only if $A_{-g}$ is invertible.
\end{nprop}
The role of $A_{-g}$ is important: every $\MD$-module admits a decomposition by eigenvalues of this term:
\begin{nlemma}[\cite{Lev}, \cite{Was}]
	\label{lem: decomposition of D module}
	Let $(M,\nabla)$ be a $\MD$-module, with $\nabla$ of order $g$. Let $\Lambda$ be a lattice of order $g \geq 2$. Let $A_{-g} := u^{g}\nabla: \Lambda/u\Lambda \rightarrow \Lambda/u\Lambda$. Then there exists a unique decomposition of $\Lambda$: \begin{equation}
		\Lambda = \bigoplus_w \Lambda_w \text{ for } w \in \KK,
	\end{equation}
	such that $u^g \nabla: \Lambda_w \rightarrow \Lambda_w$ and $\Lambda_w/u\Lambda_w = (\Lambda/u\Lambda)_w$ is the $w$-generalised eigenspace of $A_{-g}$. 
\end{nlemma}
\begin{remark}
	By tensoring with $\kk$, we thus obtain a decomposition \begin{equation}
		M = \bigoplus_w M_w.
	\end{equation}
\end{remark}
\begin{remark}
	A similar result holds for $\MD$-modules of order $g = 1$ if one additionally assumes that the eigenvalues of $A_{-1}$ do not differ by an integer.
\end{remark}
\subsection{The quantum $\MD$-module of twistor bundles}
Let $Z$ again be a twistor bundle over a hyperbolic $6$-manifold with vanishing Stiefel-Whitney classes. Let $QH^*(Z)_w$ denote the generalised eigenspaces for $c_1 \star$. Let $DQH^*(Z)$ denote the $\MD$-module $\left(QH^*(Z)((u)), \nabla_{\frac{d}{du}}\right)$. From Theorem \ref{thm: quantum cohomology of FP}, it is clear that we get a decomposition \begin{equation}
	DQH^*(Z) = DQH^*(Z)^0 \bigoplus_{i = 1}^{m} DQH^*(Z)^i,
\end{equation}
where $DQH^*(Z)^0$ is the $8$-dimensional $\MD$-module spanned by $1, \al, \al^2, \al^3, Vol_M, \dots, Vol_M\al^3$, and each $DQH^*(Z)^i$ is $4$-dimensional, spanned by $y_i, \dots, \al^3y_i$. Using Lemma \ref{lem: decomposition of D module} we can then further decompose these summands, to obtain decompositions of $\MD$-modules: \begin{equation}
	DQH^*(Z)^j = \bigoplus_w DQH^*(Z)^j_w \text{ for each } j = 0, \dots, m.
\end{equation}
For $DQH^*(Z)^0$, we find eigenvalues $\pm 2$, both with multiplicity $4$, and the summands $DQH^*(Z)^i$ have eigenvalues $\pm 2$ with multiplicities $2$.

Theorem \ref{thm: FP has unramified exponential type} is equivalent to the statement that each $DQH^*(Z)^j_w \otimes \mathcal{E}^{\frac{w}{u}}$ is regular singular for all $w$. We will first prove this for the $4$-dimensional summands $DQH^*(Z)^i$, which have bases given by $\tau^*y_i, \alpha\tau^*y_i,\al^2\tau^*y_i,\alpha^3\tau_i^*y$. With respect to this basis, the connection is given by: \begin{equation}
	\nabla = \frac{d}{du} + u^{-1}\frac{m-6}{2}Id + u^{-1}\begin{pmatrix} 0 & 0 & 0 & 0 \\ 0 & 1 & 0 & 0 \\ 0 & 0 & 2 & 0 \\ 0 & 0 & 0 & 3\end{pmatrix} + u^{-2}
	\begin{pmatrix} 
		0 & 4 & 0 & 0\\
		1 & 0 & 0 & 0\\
		0 & 1 & 0 & 4\\
		0 & 0 & 1 & 0\end{pmatrix}
\end{equation}
\begin{remark}
	Here one has to take care to write the matrix for $c_1 \star$ with respect to the basis $\tau^*y_i, \alpha\tau^*y_i,\al^2\tau^*y_i,\alpha^3\tau_i^*y$, where powers of $\al$ are powers under the classical cup product. In Theorem \ref{thm: quantum cohomology of FP}, the quantum cohomology ring is expressed in quantum powers of $\al$.
\end{remark}
After tensoring with the regular 1-dimensional $\MD$-module with connection $\frac{d}{du} + u^{-1}\frac{6-m}{2}$, we are left with the above connection without the term $u^{-1}\frac{m-6}{2}Id$. We will also denote this connection by $\nabla$. Clearly, this procedure does not affect whether $\nabla$ is of exponential type or not.

The matrix $c_1 \star$ has $2$ Jordan blocks of size $2$, with eigenvalues $\pm 2$. In particular, $c_1 \star$ is invertible, so by Proposition \ref{prop: invertible polar term gives maximum irregularity} we have $Irr(\nabla) = 4$. Another way to see this directly is to compute the differential operator associated to the cyclic vector given $y$. It is given by: \begin{equation}
	L = \partial^4 - \frac{8}{u^2}\partial^2 + \frac{16}{u^2}\partial + \left(\frac{16}{u^4} - \frac{16}{u^2}\right),
\end{equation}
so that indeed $Irr(L) = 4$. Now, write the decomposition of Lemma \ref{lem: decomposition of D module} as $DQH^*(Z)^i = M_2 \oplus M_{-2}$, where both summands have dimension $2$. Next, consider the $\MD$-module $DQH^*(Z)^i \otimes \mathcal{E}^{\frac{-2}{u}} = M_2 \otimes \mathcal{E}^{\frac{-2}{u}} \oplus M_{-2} \otimes \mathcal{E}^{\frac{-2}{u}}$. A straightforward computation (best done using a computer), yields the differential operator associated to $DQH^*(Z)^i \otimes \mathcal{E}^{\frac{-2}{u}}$ as:
\begin{equation}
	L' = \partial^4 - \frac{8}{u}\partial^{3} + \left(\frac{16}{u^{2}} + \frac{12}{u} \right)\partial^{2} - \left(\frac{32}{u^{2}} + \frac{8}{u}\right)\partial + \frac{2}{u} + \frac{12}{u^{2}}.
\end{equation}
We thus have $Irr(DQH^*(Z)^i \otimes \mathcal{E}^{\frac{-2}{u}}) = 2 = Irr(M'_{-2}) + Irr(M'_{2})$, where the latter equality holds by Lemma \ref{lem: irregularity number is additive}. Now, by applying Proposition \ref{prop: invertible polar term gives maximum irregularity}, we find $Irr(M'_{2}) = 2$, whence $Irr(M'_{-2}) = 0$, which is thus a regular summand. By symmetry (or by direct computation with $DQH^*(Z)^i \otimes \mathcal{E}^{-\frac{2}{u}}$), we find also that $DQH^*(Z)^i_{-2} \otimes \mathcal{E}^{-\frac{2}{u}}$ is regular, which shows that $DQH^*(Z)^i$ is of unramified exponential type.

For the $8$-dimensional summand $DQH^*(Z)^0$ it is not obvious that $1$ is a cyclic vector, because the space generated by $1$ under $c_1 \star$ is 7-dimensional. Interestingly though, a direct computation shows that the $\kk$-linear space generated by $1$ under $\nabla_{\frac{d}{du}}$ is $8$-dimensional, so that $1$ is a cyclic vector. We then proceed in the same way as before. Clearly $Irr(DQH^*(Z)^0)= 8$. We compute the differential operator associated to $DQH^*(Z)^0 \otimes \mathcal{E}^{\frac{2}{u}}$, which is not printed here because it is rather lengthy, and find it has irregularity number $4$. The same arguments as before apply, which finishes the proof.
	\printbibliography
\end{document}